\documentclass[amscd,amssymb,verbatim]{amsart}
\usepackage{amssymb,amsfonts,amsmath,amscd}
\usepackage[all]{xy}

\theoremstyle{plain}
\newtheorem{Thm}{Theorem}[section]

\newtheorem{Prop}[Thm]{Proposition}
\newtheorem{Cor}[Thm]{Corollary}

\theoremstyle{definition}

\theoremstyle{remark}

\newcommand{\cC}{{\mathcal C}}
\newcommand{\bbD}{\mathbb{D}}

\newcommand{\Q}{\mathbb{Q}}
\newcommand{\R}{\mathbb{R}}

\begin{document}
\title[Contractible Manifolds]{A Family of Mutually Nonhomeomorphic Separable Contractible 2-Manifolds of Cardinality $2^{2^{\aleph_0}}$}
\author{Bruce Blackadar}
\address{Department of Mathematics/0084 \\ University of Nevada, Reno \\ Reno, NV 89557, USA}
\email{bruceb@unr.edu}

\date{\today}

\maketitle
\begin{abstract}
We describe a family of open subsets $M_A$ of the Moore plane, one for each subset $A$ of $\R$.  Each of these subsets
is a separable contractible (Hausdorff) 2-manifold, which is nonmetrizable if $A$ is uncountable.  The collection $\{M_A:A\subseteq\R\}$
contains $2^{2^{\aleph_0}}$ distinct homeomorphism classes.

\end{abstract}

\section{Introduction}

There is only one contractible metrizable (paracompact, second countable) 2-manifold, namely $\R^2$.  In contrast,
we will show that there are $2^{2^{\aleph_0}}$ mutually nonhomeomorphic separable contractible nonmetrizable 2-manifolds.
We also obtain a similar family of nonseparable 2-manifolds.

The term ``(topological) manifold'' does not have a universally standard definition in the literature.  We will use the term {\em manifold}
to mean a Hausdorff locally Euclidean topological space (without boundary), not necessarily metrizable.  All examples we
will consider are (path) connected.  Manifolds are locally compact, first countable, and locally contractible (hence locally path connected).

Nonmetrizable manifolds are often regarded as pathological objects, but their existence cannot be ignored.  There are
three classic constructions of nonmetrizable manifolds: the {\em long line}, the {\em Pr\"{u}fer plane}, and the {\em Moore plane}
(the last two are closely related constructions, but the resulting manifolds have quite different properties).  We will
use a variation of the Moore plane construction, which we describe in $\S 2$.

See \cite{NyikosTheory} for a rather comprehensive discussion of the theory of nonmetrizable manifolds.  In this reference, $2^{2^{\aleph_0}}$
mutually nonhomeomorphic connected 2-manifolds are constructed; however, these examples are not separable or contractible.

I am indebted to an anonymous referee for helpful comments on the proof of Proposition \ref{Prop2}.

\section{The Moore Plane}

In this section, we describe the construction and properties of the Moore plane, originally constructed by {\sc R.\ L.\ Moore} in 1942 \cite{MooreSeparability}.
{\bf Caution:}  In some references, the name ``Moore plane'' is used to mean the {\em Nemytskii plane}, which is a different space (not a 2-manifold),
a quotient of the Moore plane.

The {\em Moore plane} $M$ is the disjoint union of the open upper half plane $\R_+^2$ and one ray
$$R_a=\{(a,z):z\geq0\}$$
for each $a\in\R$.
The upper half plane and each ray have their usual topologies; $\R_+^2$ is open, each $R_a$ is closed,
and a sequence $(x_n,y_n)$ in $\R_+^2$ converges to $(a,z)\in R_a$ if and only if
$x_n\to a$ and
$$\lim_{n\to\infty}\frac{y_n}{|x_n-a|}=\frac{1}{z}$$
i.e.\ $(x_n,y_n)$ approaches $(a,0)\in\R^2$ asymptotically along lines through $(a,0)$ with slope $\pm\frac{1}{z}$ (a vertical line if $z=0$).
The point $(a,z)$ has basic neighborhoods consisting of a small interval in $R_a$ and one (if $z=0$) or two (if $z\neq0$) small open pieces
of pie in $\R_+^2$ with vertex $(a,0)$; such a neighborhood is homeomorphic to a disk in $\R^2$, so $M$ is locally Euclidean, and obviously Hausdorff.
$M$ is clearly connected.  The upper half-plane $\R_+^2$ is separable and dense in $M$, so $M$ is separable.  However, for any $A\subseteq\R$,
$$E_A=\cup_{a\in A} R_a$$
is closed in $M$.  In particular, $\{(a,0):a\in\R\}$ is an uncountable closed subset of $M$ which is discrete in the relative topology,
so $M$ is not metrizable.  In fact, it can be shown that $E_{\Q}$ and $E_{\R\setminus\Q}$ are disjoint closed sets in $M$ which do
not have disjoint neighborhoods, so $M$ is not normal.  See e.g.\ \cite{NyikosTheory} for a further discussion of the Moore plane.

\begin{Thm}\label{Thm1}
The Moore plane is contractible.
\end{Thm}

\begin{proof}
We will define an explicit contraction.  First we contract $M$ to the subspace $Y$ consisting of $\R_+^2$ and $\{(a,0):a\in\R\}$.
This contraction will contract each $R_a$ in the usual way.  To make the contraction continuous, vertical lines in $\R_+^2$ must
be stretched near the $x$-axis in such a way that curves through $(a,0)$ have their derivatives at 0 multiplied by a suitable factor.
One formula for such a contraction is as follows, for $0\leq t\leq 1$:
$$h(t,(x,y))=\left \{ \begin{array}{ccc} (x,\sqrt{2y+t}-\sqrt{t}) & \mbox{if} & {0<y\leq 2(1-\sqrt{t})} \\ (x,y) & \mbox{if} & {y>2(1-\sqrt{t})} \end{array} \right .$$
$$h(t,(a,z))=(a,z\sqrt{t})\ .$$
To see that $h$ is continuous, note that the derivative of $\phi(y)=\sqrt{2y+t}-\sqrt{t}$ at 0 is $\frac{1}{\sqrt{t}}$; thus if $t_n\to t$
and $(x_n,y_n)\to(a,z)$, i.e.\ $x_n\to a$ and $\frac{y_n}{|x_n-a|}\to\frac{1}{z}$, and $(u_n,v_n)=h(t_n,(x_n,y_n))$, then $u_n\to a$ (in fact $u_n=x_n$)
and $\frac{v_n}{|u_n-a|}\to \frac{1}{z\sqrt{t}}$, so $h(t_n,(x_n,y_n))\to h(t,(a,z))$.  We have that $h(1,\cdot)$ is the identity on $M$
and $h(0,\cdot)$ maps $M$ onto $Y$.

It is now easy to contract $Y$ along vertical lines to the half-plane
$$H=\{(x,y):y\geq1\}$$
since $Y$ can be identified as a set with the closed
upper half-plane, although the topology at the $x$-axis is stronger than the usual topology.  Then contract the half-plane $H$ to a point.
\end{proof}

\section{A Family of Subspaces}

Let $A$ be an arbitrary subset of $\R$.  Let $M_A=M\setminus E_{A^c}$ be the subset of $M$ consisting of $\R_+^2$ and all $R_a$ for $a\in A$.
Then $M_A$ is an open subset of $M$, hence a 2-manifold.

\begin{Cor}
Each $M_A$ is separable and contractible.
\end{Cor}

\begin{proof}
Since $M_A$ contains $\R_+^2$, it is separable.  And the contraction defined in the proof of Theorem \ref{Thm1}
maps $[0,1]\times M_A$ into $M_A$ for all $A$, hence restricts to a contraction of $M_A$.
\end{proof}

The proof of the next proposition is a standard argument apparently originally due to {\sc Kuratowski} (cf.\ \cite[\S24,VI]{KuratowskiTopology}).

\begin{Prop}
There are $2^{2^{\aleph_0}}$ distinct homeomorphism classes among the $M_A$.
\end{Prop}

\begin{proof}
Fix $A\subseteq\R$.  If $B\subseteq\R$, then a homeomorphism from $M_A$ to $M_B$ can be regarded as a continuous function from
$M_A$ to $M$.  But $M_A$ is separable and each continuous function from $M_A$ to $M$ is completely determined by its values
on a countable dense subset of $M_A$.  It is easily checked that the cardinality of $M$ is $2^{\aleph_0}$ (in fact the cardinality
of any connected manifold is $2^{\aleph_0}$), so there are only $(2^{\aleph_0})^{\aleph_0}=2^{\aleph_0}$ continuous functions
from $M_A$ to $M$.  Thus there are at most $2^{\aleph_0}$ subsets $B$ of $\R$ for which $M_B$ is homeomorphic to $M_A$.
Since there are $2^{2^{\aleph_0}}$ subsets of $\R$, there are $2^{2^{\aleph_0}}$ distinct homeomorphism classes.
\end{proof}

The $M_A$ are not all topologically distinct.  If $A$ is countable, it is easy to see that $M_A\cong\R^2$, and conversely.  It seems likely that
if the symmetric difference $A\Delta B$ is countable, then $M_A\cong M_B$; the converse does not hold, since if $A$ and $B$ are subsets of $\R$
which are conjugate by a $\cC^2$-diffeomorphism of $\R$, it is clear that $M_A\cong M_B$ (this is unclear if $A$ and $B$ are just conjugate by a homeomorphism,
since the vertical lines have to be scaled to preserve the asymptotic direction of approach to the points on the rays; even a $\cC^1$-diffeomorphism
might not be good enough for technical reasons).

\section{Extensions and Variations}

We can extend the class of examples of the last section.  Let $N$ be the space obtained from the open unit disk $\bbD$ by adding rays at
each point of the boundary circle in the same manner as in the construction of $M$.  Then $N$ is a separable nonmetrizable 2-manifold,
and it can be shown by an argument similar to the one
in the proof of Theorem \ref{Thm1} that $N$ is contractible.  If $A$ is any subset of the unit circle, an open submanifold $N_A$ can
be formed by adding rays only at the points of $A$.  Then $M\cong N_I$, where $I$ is a proper open interval in the unit circle;
more generally, if $B\subseteq\R$, then $M_B\cong N_A$ for the corresponding subset $A$ of $I$.  So $N$ is the ``largest'' member of this
extended family.

We can also do a similar construction in the Pr\"{u}fer plane $P$, which consists of adding to the open upper half plane a copy
of the closed lower half plane at each point of the $x$-axis.  This $P$ is a 2-manifold which is contractible \cite[Appendix A, Problem 6]{SpivakDifferential}
but not separable.  If $A\subseteq\R$, we can form $P_A$ by adding the lower half-planes only at points of $A$.  Each of these is
contractible in the same way.  $P_A$ is separable if and only if $A$ is countable, and in this case is homeomorphic to $\R^2$.

\begin{Prop}\label{Prop2}
There are $2^{2^{\aleph_0}}$ distinct homeomorphism classes among the $P_A$.
\end{Prop}

\begin{proof}
Fix $A\subseteq\R$.  If $B\subseteq\R$, then a homeomorphism from $P_A$ to $P_B$ can be regarded as a continuous function from
$P_A$ to $P$.  The image of the upper half plane under any such map can intersect only countably many of the lower half planes
since the upper half plane has the countable chain condition.  Thus the image of the upper half plane is contained in $P_C$
for some countable subset $C$ of $\R$.  But there are $2^{\aleph_0}$ countable subsets of $\R$, and only $2^{\aleph_0}$
continuous maps from the upper half plane to each such $P_C$ by the Kuratowski argument, so there can only be $2^{\aleph_0}$
homeomorphisms from $P_A$ to some $P_B$.
\end{proof}

As with the Moore plane, we can expand the class of examples by adding to $\bbD$ a copy of the closed lower half plane at each boundary
point (in the picture of $N$ described in \cite{NyikosTheory}, this corresponds to adding a wedge instead of a ray at the boundary point)
to obtain a 2-manifold $Q$, which is again contractible.  If $B$ is any subset of the circle, we can just add wedges at points of $B$.
Then $P\cong Q_I$, where $I$ is a proper open subinterval of the circle; more generally, each $P_A$ is homeomorphic to a $Q_B$.

We can put all the examples together into one family.  If $A$ and $B$ are disjoint subsets of the circle, to $\bbD$ we can add rays at
the points of $A$ and wedges at the points of $B$ to obtain a 2-manifold $D_{A,B}$.  We have $N_A=D_{A,\emptyset}$ and $Q_B=D_{\emptyset,B}$.
The stretching of $\bbD$ near the boundary needed to contract a wedge (i.e.\ the stretching of the upper half plane needed to
contract $P$) can be taken to be the same as the stretching needed for a ray (i.e.\ the stretching of the upper half plane
described in the proof of Theorem \ref{Thm1}); thus $D_{A,B}$ is contractible.  $D_{A,B}$ is separable if and only if
$B$ is countable.  The homeomorphism classes of the $D_{A,B}$ include $2^{2^{\aleph_0}}$ distinct homeomorphism classes
different from all the $N_A$ and $Q_B$.  There is no ``largest'' manifold in this family.

\bibliography{mancon2ref}

\begin{thebibliography}{Moo42}

\bibitem[Kur66]{KuratowskiTopology}
K.~Kuratowski.
\newblock {\em Topology. {V}ol. {I}}.
\newblock New edition, revised and augmented. Translated from the French by J.
  Jaworowski. Academic Press, New York, 1966.

\bibitem[Moo42]{MooreSeparability}
R.~L. Moore.
\newblock Concerning separability.
\newblock {\em Proc. Nat. Acad. Sci. U. S. A.}, 28:56--58, 1942.

\bibitem[Nyi84]{NyikosTheory}
Peter Nyikos.
\newblock The theory of nonmetrizable manifolds.
\newblock In {\em Handbook of set-theoretic topology}, pages 633--684.
  North-Holland, Amsterdam, 1984.

\bibitem[Spi79]{SpivakDifferential}
Michael Spivak.
\newblock {\em A comprehensive introduction to differential geometry. {V}ol.
  {I}}.
\newblock Publish or Perish Inc., Wilmington, Del., second edition, 1979.

\end{thebibliography}
\bibliographystyle{alpha}

\end{document}